\documentclass[14pt]{amsart}
\address{\newline{\normalsize Courant Institute, NYU, 251 Mercer str., New York, NY 10012, USA}\newline{\it E-mail address}:
karzhema@cims.nyu.edu}
\usepackage{amscd,amssymb}
\usepackage{amsthm,amsmath,amssymb}
\usepackage[matrix,arrow]{xy}

\makeatletter\@addtoreset{equation}{section}\makeatother

\makeatletter\@addtoreset{subsection}{equation}\makeatother

\newtheorem{theorem}[equation]{Theorem}
\newtheorem{proposition}[equation]{Proposition}
\newtheorem{lemma}[equation]{Lemma}
\newtheorem{corollary}[equation]{Corollary}

\newtheorem{question}[equation]{Question}
\newtheorem{theorem-definition}[equation]{Theorem-definition}

\theoremstyle{definition}
\newtheorem{example}[equation]{Example}
\newtheorem{definition}[equation]{Definition}

\theoremstyle{remark}
\newtheorem{remark}[equation]{Remark}

\newcommand{\fie}{{\bf k}}

\begin{document}

\Large

\title{One construction of a K3 surface with a dense set of
rational points}

\author{Ilya Karzhemanov}

\begin{abstract}
We prove that there exists a number field $\fie$ and a smooth
projective $\mathrm{K3}$ surface $S_{22}$ (of genus $12$) over
$\fie$ such that the geometric Picard number of $S_{22}$ is equal
to $1$ and the $\fie$-rational points of $S_{22}$ are Zariski
dense.
\end{abstract}

\sloppy

\maketitle

\bigskip

\section{Introduction}
\label{section:introduction}

The present paper is concerned with the density problem for the
set of rational points on $\mathrm{K3}$ surfaces defined over a
number field. The following is the main question in the area:

\begin{question}
\label{theorem:potential-density-problem} Let $\frak{S}$ be a
smooth projective $\mathrm{K3}$ surface over a number field $\fie$
(this situation is denoted as $\frak{S}\slash\fie$). Are the
$\fie$-rational points (or simply \emph{$\fie$-points} for short)
of $\frak{S}$ potentially dense? In other words, is the set
$\frak{S}(\fie)$ of $\fie$-points Zariski dense in $\frak{S}$,
after possibly replacing $\fie$ with its finite extension?
\end{question}

Question~\ref{theorem:potential-density-problem} is a version of
modified \emph{Weak Lang's Conjecture} (see \cite[Conjecture
1.6]{harris-tschinkel}) and is expected to always have the
positive answer (see \cite{abromovich}, \cite{hassett},
\cite{tschinkel} and references therein for excellent surveys on
the density problem). Unfortunately, only the partial evidence for
both possible answers to
Question~\ref{theorem:potential-density-problem} has been obtained
so far (see \cite{bogomolov-tschinkel}, \cite{ellenberg},
\cite{hassett}, \cite{van-luijk}), in that one did not have yet a
single example of $\frak{S}$ (and $\fie$) with the geometric
Picard number $1$ and Zariski dense set $\frak{S}(\fie)$.

The present paper aims to eliminate the latter defect:

\begin{theorem}
\label{theorem:main} There exists a number field $\fie$ and a
smooth projective $\mathrm{K3}$ surface $S_{22}\slash\fie$ such
that the geometric Picard number of $S_{22}$ is $1$, $S_{22}$ is
of genus $12$, and the set $S_{22}(\fie)$ is Zariski dense in
$S_{22}$.\footnote{The proof of Theorem~\ref{theorem:main}, as
will be described shortly, might allow one, in principle,
establish the same claim for all \emph{BN-general}
$S_{22}\slash\fie$ (cf. Definition~\ref{definition:bn-general}
below), including those with geometric Picard number $1$ of
course. There are, however, certain technical difficulties with
this generalization, and we are not going to address these at
present (but see some discussion in
Remark~\ref{remark:test-set-of-sections}).}
\end{theorem}

Recall that the genus assumption in Theorem~\ref{theorem:main}
implies the surface $S_{22}$ can be embedded into the projective
space $\mathbb{P}^{12}$ as a subvariety of degree $22$.

Let us briefly describe our approach towards the proof of
Theorem~\ref{theorem:main}. Firstly, there exists a $\mathrm{K3}$
surface $S$ whose (geometric) Picard lattice $\mathrm{Pic}\,S$ is
generated by a very ample divisor $H$ and a $(-2)$-curve $C_0$,
satisfying $(H^2) = 70$ and $H \cdot C_0 = 2$ (see
Proposition~\ref{theorem:r-polarized} below). In particular, the
pair $(S,H)$ represents a \emph{polarized $\mathrm{K3}$ surface of
genus $36$} (cf. Definition~\ref{definition:bn-general}). Further,
the divisor $H - 4C_0$ also provides a polarization on $S$, of
genus $12$ (see Lemma~\ref{theorem:polarizations}), and the
surface $(S, H - 4C_{0})$ turns out to be \emph{BN general} (see
Lemma~\ref{theorem:bn-general}). In particular, due to S.\,Mukai
(whose results we recall in
Theorem-definition~\ref{theorem:existence-if-rigid-v-b} and
Remark~\ref{remark:s-12-is-unirational}), there exists a
\emph{rigid} vector bundle $E_3$ on $S$ of rank $3$, such that
$\dim H^{0}(S, E_{3}) = 7$ and the first Chern class of $E_3$
equals $H - 4C_0$. Note that such $E_3$ is unique and determines a
morphism $\Phi_{\scriptscriptstyle E_{3}} : S \longrightarrow G(3,
7)$ into the Grassmannian $G(3, 7) \subset
\mathbb{P}(\bigwedge^{3}\mathbb{C}^{7})$ (cf.
Remark~\ref{remark:when-e-gives-a-morphism}). Moreover,
$\Phi_{\scriptscriptstyle E_{3}}$ coincides with the embedding $S
\hookrightarrow \mathbb{P}^{12}$ given by the linear system $|H -
4C_{0}|$, and the surface $S = \Phi_{E_{3}}(S) \subset G(3, 7)
\cap \mathbb{P}^{12}$ can be described by explicit equations on
$G(3, 7)$ (see Theorem~\ref{theorem:mukai}).

All these constructions apply without change to any \emph{general}
(as a point in moduli) polarized $\mathrm{K3}$ surface
$(S_{22},L_{22})$ of genus $12$ (cf.
Example~\ref{example:ex-of-bn-gen}). Moreover, all these
constructions can be executed over an appropriate number field
$\fie$, which makes one able to use the geometric description of
surfaces $S$ and $S_{22}$ (now defined over $\fie$) to study their
arithmetical properties (see {\ref{subsection:pre-223}} below for
details). The first observation in this way is that the set of
$\fie$-points is Zariski dense in $S$ (see
Proposition~\ref{theorem:my-potential-density}).

\begin{remark}
\label{remark:phenomenon-for-k-3-with-small-picard} One easily
checks that the surface $S$ does not admit any polarization of
genus $ < 12$. At the same time, $S$ is one of the examples of
smooth projective $\mathrm{K3}$ surfaces over $\fie$, containing a
smooth rational curve and Zariski dense set of $\fie$-points.
Other examples include the minimal resolution of the double cover
of $\mathbb{P}^2$ ramified in a general sextic with one node
(genus $2$ case) and a general quartic surface in $\mathbb{P}^3$
with a line (genus $3$ case); these two types of surfaces also
possess the density property of $\fie$-points (see
\cite[Proposition 3.1]{bogomolov-tschinkel-1}, \cite[Theorem
1.5]{harris-tschinkel}) and we hope this phenomenon can be
exploited further. For instance, modifying our arguments from the
proof of Theorem~\ref{theorem:main} appropriately, one could try
to settle down Question~\ref{theorem:potential-density-problem} in
the case of lower genera (starting with e.g. $\le 3$). On the
other hand, let us mention that there are smooth projective
$\mathrm{K3}$ surfaces over $\fie$, having geometric Picard number
$\ge 2$ and Zariski dense set of $\fie$-points, but without any
$(-2)$-curves (see \cite{silverman}). Heuristically, this setting
is different from the one we have started this Remark with, and so
one would probably need another method here to attack
Question~\ref{theorem:potential-density-problem}.
\end{remark}

Further, the main technical result used in the proof of
Theorem~\ref{theorem:main}, namely that
Proposition~\ref{theorem:many-seFctions} below, was inspired by
the following:

\begin{theorem}[{see \cite[Theorem 1]{hassett-tschinkel}}]
\label{theorem:hassett-tschinkel-density-over-k-t} Let $B$ be a
complex curve and $F := \mathbb{C}(B)$ its function field. There
exist non-isotrivial smooth $\mathrm{K3}$ surfaces over $F$ of any
genus, varying between $2$ and $10$, which have geometric Picard
number $1$ and Zariski dense set of $F$-points.
\end{theorem}

Though our arguments are different from those in paper
\cite{hassett-tschinkel}, -- e.g. the proof of
Proposition~\ref{theorem:many-seFctions} is more or less a direct
parameter count, -- the idea of constructing the example we need
over a function field first was of great value for us.

More specifically, working with the Grassmannian $G(3, 7)$ over
the function field $F := \fie(\mathbb{P}^{1})$, we construct a
smooth projective $\mathrm{K3}$ surface $\mathcal{S}$ over $F$
(see the end of {\ref{subsection:pr-2}}) such that its general
\emph{constant} hyperplane section $\mathcal{C}$ over $F$ contains
infinitely many $\overline{\mathbb{Q}}(\mathbb{P}^{1})$-points
(see Corollary~\ref{theorem:many-sections-cor}). Moreover,
applying the functional version of the Mordell Conjecture (see
Theorem~\ref{theorem:manin-mordell-weil}), we obtain that every
such point on $\mathcal{S}$ is defined over $F$, once its
specialization is defined over $\fie$ (see
Lemmas~\ref{theorem:all-sections-are-over-k}). This together with
the previously stated density property for $S(\fie)$ implies that
$\mathcal{S}$ contains a dense set of $F$-points (cf.
Remark~\ref{remark:every-thing-in-U}). In addition, the geometric
Picard number of $\mathcal{S}$ equals $1$ (see
Lemma~\ref{theorem:pic-calc}), while the genus of $\mathcal{S}$ is
$12$ (compare with
Theorem~\ref{theorem:hassett-tschinkel-density-over-k-t} above).
Finally, playing with $\mathcal{S}$ and its specialization (at
infinity), we arrive at the surface as in
Theorem~\ref{theorem:main} (see Lemmas~\ref{theorem:1-to-1-lemma}
and \ref{theorem:1-to-2-lemma}).

\bigskip

\thanks{{\bf Acknowledgments.}
The text had gone a long way of revisions and discussion. I am
grateful to many people, especially to F. Bogomolov, C. Liedtke,
A. Lopez, B. Poonen, E. Sernesi, S. Tanimoto and various anonymous
referees, who have helped me with improving on the exposition.
Main parts of the paper were written during my stays at the
\emph{Courant Institute}, New York (US), \emph{Lorentz Center},
Leiden (Netherlands) and \emph{I.H.E.S.}, Bures-sur-Yvette
(France). I am grateful to these institutions for hospitality and
for providing excellent working conditions. The work was partially
supported by the \emph{Project TROPGEO of the European Researh
Council}}, by the \emph{World Premier International Research
Initiative (WPI), MEXT, Japan}, and by the \emph{Grant-in-Aid for
Scientific Research (26887009) from Japan Mathematical Society
(Kakenhi)}.

\bigskip

\section{Preliminaries}
\label{section:preliminaries}

\refstepcounter{equation}
\subsection{}

Let us formulate several auxiliary notions and results which we
will use in the proof of Theorem~\ref{theorem:main}. We start with
Geometry first (over the ground field $\mathbb{C}$):

\begin{definition}[cf. {\cite[Definition 3.8]{mukai}}]
\label{definition:bn-general} Let $\frak{S}$ be a smooth
projective $\mathrm{K3}$ surface and $L$ a primitive polarization
on $\frak{S}$. By this we mean: $L = \mathcal{O}_{\frak{S}}(1)$
for some projective embedding $\frak{S} \subset \mathbb{P}^{g}$,
where $g := (L^2)/2 + 1$ is the \emph{genus} of $\frak{S}$, and
$L$ is primitive as a vector in the lattice $H^{2}(\frak{S},
\mathbb{Z})$. Then the polarized $\mathrm{K3}$ surface
$(\frak{S},L)$ is called \emph{BN general}, if
$h^{0}(\frak{S},L_{1})h^{0}(\frak{S},L_{2}) < g + 1$ for all
non-trivial line bundles $L_{1},L_{2} \in \mathrm{Pic}\,\frak{S}$
such that $L = L_1 + L_2$.
\end{definition}

\begin{example}
\label{example:ex-of-bn-gen} Let $\mathcal{K}_g$ be the moduli
space of primitively polarized $\mathrm{K3}$ surfaces of genus
$g$. Then generic $\mathrm{K3}$ surface in $\mathcal{K}_g$ is BN
general. Note also that BN general $\mathrm{K3}$ surfaces form a
Zariski open subset in $\mathcal{K}_g$.
\end{example}

\begin{definition}
\label{definition:generated-by-global-sections} Let $W$ be a
smooth projective variety and $E$ a vector bundle on $W$. Recall
that $E$ is said to be \emph{generated by global sections} when
the natural homomorphism of $\mathcal{O}_W$-modules
$ev_{\scriptscriptstyle E} : H^{0}(W,E) \otimes \mathcal{O}_W
\longrightarrow E$ is surjective (we identify $E$ with its sheaf
of sections).
\end{definition}

\begin{remark}[cf. {\cite[Section 2]{mukai-1}}]
\label{remark:when-e-gives-a-morphism} In the notation of
Definition~\ref{definition:generated-by-global-sections}, if $E$
is generated by global sections, then one arrives at a natural
morphism $\Phi_{\scriptscriptstyle E} : W \longrightarrow G(r,N)$;
here $r := \mathrm{rank}\,E$, $N := h^{0}(W,E)$ and $G(r,N)$ is
the Grassmannian of $r$-dimensional linear subspaces in
$\mathbb{C}^N$. The morphism $\Phi_{\scriptscriptstyle E}$ sends
each $x \in W$ to the subspace $E_{x}^{\vee}\subset
H^{0}(W,E)^{\vee}$ dual to the fiber $E_x \subset E$. In
particular, we have an equality $E = \Phi_{\scriptscriptstyle
E}^{*}\mathcal{E}_r$ for the universal quotient vector bundle
$\mathcal{E}_r$ on $G(r,N)$, so that $H^{0}(W,E) =
H^{0}(G(r,N),\mathcal{E}_r)$. Furthermore, if the natural
homomorphism $\bigwedge^{r}H^{0}(W,E) \longrightarrow
H^{0}(W,\bigwedge^{r}E)$ (induced by the $r$-th exterior power of
$ev_{\scriptscriptstyle E}$) is surjective, then
$\Phi_{\scriptscriptstyle E}$ coincides with the embedding
$\Phi_{\scriptscriptstyle|c_1(E)|} : W \hookrightarrow \mathbb{P}
:= \mathbb{P}(H^{0}(W,L)^{\vee})$, given by the linear system
$\left|c_1(E)\right|$. More precisely, the diagram
$$
\begin{array}{rcl}
\Phi_{\scriptscriptstyle E} : W&\longrightarrow G(r,N)\cap\mathbb{P}&\subset G(r,N)\\
\cap&&\qquad\cap\\
\mathbb{P}&\hookrightarrow&
\mathbb{P}(\bigwedge^{r}H^{0}(W,E)^{\vee})
\end{array}
$$
is commutative, where $G(r,N)$ is considered with respect to its
Pl\"ucker embedding into
$\mathbb{P}(\bigwedge^{r}H^{0}(W,E)^{\vee})$.
\end{remark}

\begin{theorem-definition}
\label{theorem:existence-if-rigid-v-b} Let $(\frak{S},L)$ be a BN
general polarized $\mathrm{K3}$ surface of genus $g$. Then for
every pair of integers $(r,s)$, with $g = rs$, there exists a
(\emph{Gieseker}) stable vector bundle $E_r$ on $\frak{S}$ of rank
$r$ and such that the following holds:

\begin{enumerate}

\item\label{1-it} $c_1(E_r) = L$;

\smallskip

\item\label{2-it} $H^{i}(\frak{S},E_{r}) = 0$ for all $i > 0$ and $h^{0}(\frak{S},E_{r}) = r + s$;

\smallskip

\item\label{3-it} $E_r$ is generated by global sections and the natural homomorphism $\bigwedge^{r}H^{0}(\frak{S},E_{r})
\longrightarrow H^{0}(\frak{S},\bigwedge^{r}E_{r}) =
H^{0}(\frak{S},L)$ is surjective;
\smallskip

\item\label{4-it} any stable vector bundle on $\frak{S}$, which
satisfies $(\ref{1-it})$ and $(\ref{2-it})$, is isomorphic to
$E_r$.

\end{enumerate}

$E_r$ is called the \emph{rigid} vector bundle.
\end{theorem-definition}

\begin{proof}[Sketch of the proof]
Statements \eqref{1-it} -- \eqref{4-it} follow essentially from
results and discussion in \cite{mukai-1}, \cite{mukai},
\cite{mukai-7} and \cite{mukai-8} (cf. \cite[Theorem 3]{mukai-1}
and \cite[Section 4]{mukai}). Namely, there always exists a
\emph{semi-stable} rigid \emph{coherent sheaf} $E_r$, with
$\chi(\frak{S},E_r) = r + s$ and $c_1(E_r) = L$, as one sees from
e.g. \cite[Corollary 0.2]{mukai-8} and the assumptions on $g,r,s$.
Then BN generality of $\frak{S}$ is applied to see that $E_r$ is
stable. Indeed, otherwise there exists a subsheaf $E'
\varsubsetneq E_r$ of rank $r' \le r$, satisfying
$$\frac{h^0(\frak{S},c_1(E'))}{r'} = \frac{h^0(\frak{S},L)}{r} =
\frac{rs + 1}{r} = s + \frac{1}{r};$$ the latter easily gives $r =
r'$ and contradiction with $h^0(\frak{S},c_1(E')) < g + 1$ (cf.
Definition~\ref{definition:bn-general}).

Uniqueness (a.k.a. \emph{simplicity}) of $E_r$ follows from
\cite[Corollary 3.5]{mukai-7}. In turn, $E_r$ is locally free due
to \cite[Theorem 5.1 and Proposition 3.3]{mukai-7}, which gives
the statement \eqref{4-it}. Further, we have
$$H^2(\frak{S},E_r) = H^0(\frak{S},E_r^{\vee}) = 0$$ by stability (and Serre duality),
for otherwise there would be a non-trivial morphism of sheaves
$E_r \longrightarrow \mathcal{O}_{\frak{S}}$. From this and the
estimate
$$h^0(\frak{S},E_r) + h^0(\frak{S},E_r^{\vee}) \ge
\chi(\frak{S},E_r) = r + s$$ one gets \eqref{2-it}. Let us finally
establish \eqref{3-it}.

Put $E := E_r,W := \frak{S}$ in the notation of
Remark~\ref{remark:when-e-gives-a-morphism} and assume that
$\mathrm{Pic}\,\frak{S} = \mathbb{Z} \cdot L$, for the BN general
case will follow from Definition~\ref{definition:bn-general} and a
routine argument involving the Hilbert scheme of $\frak{S}$.

Fix some basis $s_1,\ldots,s_{r + s}$ of $H^0(\frak{S},E_r)$ (we
do not assume right now $E_r$ to be generated by global sections).
Note that various products $s_{i_1} \wedge s_{i_2}
\wedge\ldots\wedge s_{i_r}$, $i_1 < i_2 < \ldots < i_r$, naturally
define global sections of the line bundle $c_1(E_r) = L$. Suppose
that there is a codimension $1$ locus in $\frak{S}$ on which all
these sections vanish. Then, since $\mathrm{Pic}\,\frak{S} =
\mathbb{Z} \cdot L$, the morphism $\Phi_{\scriptscriptstyle E_r}$
must be constant by construction on an open subset
$\subseteq\frak{S}$, which is impossible. Hence the $r$-th wedge
products of different $s_i$ do not vanish simultaneously along any
curve on $\frak{S}$.

Thus, in particular, the sections $s_i$ generate a subsheaf
$\mathcal{F} \subseteq E_r$, having $c_1(\mathcal{F}) = c_1(E_r)$.
This yields $\mathcal{F} = E_r$ (because of stability) and shows
that $E_r$ is generated by global sections.

Now, if $C \sim L$ is a general curve on $\frak{S} =
\Phi_{\scriptscriptstyle L}(\frak{S})$ and $x \in C$ is a fixed
point, then the morphism $\Phi_{\scriptscriptstyle E_r}$ is not
smooth at $x$ iff the differentials of $s_{i_1} \wedge s_{i_2}
\wedge\ldots\wedge s_{i_r}$ vanish at $x$. This implies that the
set of all possible degenerate points of $\Phi_{\scriptscriptstyle
E_r}\big\vert_C$ coincides with the zero locus of some section
from $H^0(C,L^{\vee}\otimes K_C)$. But $L^{\vee}\otimes K_C = -L +
K_C = \mathcal{O}_C$. Hence $\Phi_{\scriptscriptstyle
E_r}\big\vert_C$ is everywhere smooth on $C$.

Drawing $C$ through every point $x \in \frak{S}$ we obtain that
the morphism $\Phi_{\scriptscriptstyle E_r}$ is smooth as well.
Since $\frak{S}$ is simply-connected, this is only possible when
$\Phi_{\scriptscriptstyle E_r}$ is an embedding, and so one gets
$E_r = \Phi_{\scriptscriptstyle E_r}^*\mathcal{E}_r$. Surjectivity
of $\bigwedge^{r}H^{0}(\frak{S},E_{r}) \longrightarrow
H^{0}(\frak{S},\bigwedge^{r}E_{r})$ is now clear and \eqref{3-it}
follows.
\end{proof}

\refstepcounter{equation}
\subsection{}
\label{subsection:pre-1}

Consider the Fano threefold $X$ with canonical Gorenstein
singularities and degree $(-K_{X})^3 = 70$ (see \cite{isk-prok},
\cite{karz-2}, \cite{karz-1}). Recall that the divisor $-K_{X}$ is
very ample and the linear system $|-K_{X}|$ provides an embedding
of $X$ into $\mathbb{P}^{37}$. Furthermore, we have
$$
\mathrm{Pic}\,X = \mathbb{Z}\cdot K_X \qquad \mbox{and} \qquad
\mathrm{Cl}\,X = \mathbb{Z}\cdot K_X \oplus \mathbb{Z}\cdot
\hat{E},
$$
where $\hat{E}$ is the quadratic cone on $X \subset
\mathbb{P}^{37}$ (see \cite[Corollary 3.11]{karz-3}). The
following result was proved in \cite{karz-3}:

\begin{proposition}[see {\cite[Corollary 1.5]{karz-3}}]
\label{theorem:r-polarized} The general element $S \in |-K_{X}|$
is a $\mathrm{K3}$ surface such that the lattice $\mathrm{Pic}\,S$
is generated by the very ample divisor $H \sim -K_{X}\big\vert_S$
and the $(-2)$-curve $C_0 := \hat{E}\big\vert_S$. One also has
$(H^2) = 70$ and $H \cdot C_0 = 2$.
\end{proposition}

Let $S$ be the $\mathrm{K3}$ surface as in
Proposition~\ref{theorem:r-polarized}.

\begin{lemma}
\label{theorem:polarizations} The divisor $H - 4C_0$ is ample.
\end{lemma}

\begin{proof}
Suppose $Z\subset S$ is an irreducible curve such that $(H -
4C_{0}) \cdot Z \le 0$. Then we have $Z \ne C_{0}$.\footnote{In
the following argument, we are crucially relying on the fact that
$S \in |-K_X|$, since it ensures for instance both classes $H$ and
$C_0$ are \emph{effective}; this is not at all clear without the
presence of $X \supset S$ as above.} Write
$$
Z = aH + bC_{0}
$$
in $\mathrm{Pic}\,S$ for some $a,b \in \mathbb{Z}$. Note that $a
> 0$ because the linear system $|m(H + C_{0})|$ is basepoint-free for
$m \gg 1$ (it provides a contraction of the $(-2)$-curve $C_{0}$)
and $(H + C_{0}) \cdot Z = 72a$. On the other hand, we have
$$
0 \ge (H - 4C_{0}) \cdot Z = 62a + 10b,
$$
which implies that $b < -6a$. But then we get
$$
(Z^{2}) = 70a^2 + 4ab - 2b^2 \le -26a^2 < -2,
$$
a contradiction. Thus we obtain $(H - 4C_{0}) \cdot Z
> 0$ for every curve $Z \subset S$. Then $H - 4C_{0}$ is ample by
the Nakai-Moishezon criterion (note that $(H - 4C_{0})^2 = 22$).
\end{proof}

\begin{lemma}
\label{theorem:bn-general} The polarized $\mathrm{K3}$ surface
$(S,H - 4C_{0})$ (of genus $12$) is BN general.
\end{lemma}

\begin{proof}
Suppose that
$$
H - 4C_0 = L_1 + L_2
$$
for some non-trivial $L_{1},L_{2} \in \mathrm{Pic}\,S$. We may
assume that both $h^{0}(S,L_{1}),h^{0}(S,L_{2}) > 0$. Write
$$
L_{i} = a_i H + b_i C_0
$$
in $\mathrm{Pic}\,S$ for some $a_{i},b_{i} \in \mathbb{Z}$. Note
that $a_i \ge 0$ (cf. the proof of
Lemma~\ref{theorem:polarizations}), and hence we get $a_1 = 1,a_2
= 0$, say. This implies that $b_2 \ne 0$. Now, if $b_2 < 0$, then
$h^{0}(S,L_{2}) = 0$ and we are done. Finally, if $b_2
> 0$, then $b_{1} \le -5$ and so
$$
h^{0}(S,L_{1})h^{0}(S,L_{2}) = h^{0}(S,H + b_{1}C_{0}) < h^{0}(S,H
- 4C_{0}) = 13,
$$
since $h^{0}(S,L_{2}) = h^{0}(S,b_{2}C_{0}) = 1$.
\end{proof}

Lemmas~\ref{theorem:polarizations}, \ref{theorem:bn-general} and
Theorem-definition~\ref{theorem:existence-if-rigid-v-b} imply that
there exists a rigid vector bundle $E_3$ on $S$ of rank $3$,
satisfying $c_1(E_3) = H - 4C_0$ and $h^{0}(S,E_{3}) = 7$. Then
from Remark~\ref{remark:when-e-gives-a-morphism} we get the
morphism $\Phi_{\scriptscriptstyle E_{3}} : S \longrightarrow
G(3,7)\cap \mathbb{P}^{12} \subset
\mathbb{P}(\bigwedge^{3}\mathbb{C}^{7})$ which coincides with the
embedding $\Phi_{\scriptscriptstyle|H - 4C_{0}|} : S
\hookrightarrow \mathbb{P}^{12}$. One also has $E_3 =
\Phi_{\scriptscriptstyle E_{3}}^{*}\mathcal{E}_{3}$ for the
universal quotient vector bundle $\mathcal{E}_{3}$ on $G(3,7)$.

We now recall an explicit description of the image
$\Phi_{\scriptscriptstyle E_{3}}(S)$. Namely, let $\Lambda$ be
some global section of $\bigwedge^{3}\mathcal{E}_{3}$, which one
treats as a skew-form $\in\bigwedge^{3}\mathbb{C}^{7}\simeq
H^{0}(G(3, 7), \bigwedge^{3}\mathcal{E}_{3})$ (note also that
$\Lambda$ corresponds to a hyperplane section of $G(3,7)$ in the
Pl\"ucker embedding). Similarly, consider some $\sigma_{1},
\sigma_{2}, \sigma_{3} \in H^{0}(G(3, 7),
\bigwedge^{2}\mathcal{E}_{3}) \simeq \bigwedge^{2}\mathbb{C}^{7}$
and put $\Sigma_0 := (\sigma_{1}, \sigma_{2}, \sigma_{3}) \in
H^{0}(G(3, 7), (\bigwedge^{2}\mathcal{E}_{3})^{\oplus 3})$. Then
the following holds:

\begin{theorem}[see {\cite[Theorem 5.5]{mukai}}]
\label{theorem:mukai} The surface $(S,H - 4C_0) =
\Phi_{\scriptscriptstyle E_{3}}(S) \subset \mathbb{P}^{12}$
coincides with the locus
$$
G(3, 7) \cap (\Lambda = 0) \cap (\Sigma_0 = 0) \subset
\mathbb{P}(\bigwedge^{3}\mathbb{C}^{7}).
$$
\end{theorem}

\begin{remark}[cf. \cite{mukai-1}, \cite{mukai}, \cite{mukai-6}]
\label{remark:s-12-is-unirational} One repeats the previous
arguments literally in the case of any BN general polarized
$\mathrm{K3}$ surface $(S_{22},L_{22})$ of genus $12$. Namely,
$S_{22}$ can be embedded into $G(3,7) \cap \mathbb{P}^{12}$, where
it coincides with the locus $G(3, 7) \cap (\Lambda_{22} = 0) \cap
(\tau_1 = \tau_2 = \tau_3 = 0)$ for some $\Lambda_{22} \in
\bigwedge^{3}\mathbb{C}^7$ and $\tau_1, \tau_2, \tau_3 \in
\bigwedge^{2}\mathbb{C}^7$ (so that $\mathcal{O}_{S_{22}}(L_{22})
\simeq \mathcal{O}_{G(3,7)}(1)\big\vert_{S_{22}}$). Conversely,
any such locus defines, for generic $\Lambda_{22},\tau_i$, a BN
general polarized $\mathrm{K3}$ surface of genus $12$ (cf.
Example~\ref{example:ex-of-bn-gen}). One can also prove that
$(S_{22},L_{22})$ is uniquely determined by the
$\text{PGL}(7,\mathbb{C})$-orbits of $\Lambda_{22}$ and $\tau_i$.
This delivers a birational map between $\mathcal{K}_{12}$ and a
$\mathbb{P}^{13}$-bundle over the orbit space $\frak{M} :=
G(3,\bigwedge^{2}\mathbb{C}^{7})//\text{PGL}(7,\mathbb{C})$. In
particular, the triple $(\tau_1, \tau_2, \tau_3)$ (resp.
$\Lambda_{22}$) corresponds to a generic point in $\frak{M}$
(resp. in the fiber $\mathbb{P}^{13}$), while $\Sigma_0$
corresponds to a generic point in some codimension $1$ locus
$\frak{M}_0 \subset \frak{M}$ (see \cite[Corollary 1.5]{karz-3}).
\end{remark}

\refstepcounter{equation}
\subsection{}
\label{subsection:pre-223}

We now turn to Arithmetics. Fix some number field $\fie \subset
\overline{\mathbb{Q}}$ once and for all (although we will allow
replacements of $\fie$ with its finite degree extensions when
needed).

Recall that the threefold $X$ from {\ref{subsection:pre-1}} is
constructed as follows: one takes the weighted projective space
$\mathbb{P} := \mathbb{P}(1,1,4,6)$, embeds $\mathbb{P}$
anticanonically inside $\mathbb{P}^{38}$, picks \emph{any}
singular $\mathrm{cDV}$ point $O \in \mathbb{P}$ and projects
$\mathbb{P}$ from $O$; the image of $\mathbb{P}$ under this
projection is precisely our $X$ (see e.g. \cite[Example
3.20]{karz-2}). This construction can obviously be carried over
$\fie$.

The $\mathrm{K3}$ surface $S \in |-K_X|$ can also be defined over
$\fie$. More precisely, applying Lemma~\ref{theorem:bn-general} to
the general point of $|-K_X|$, from
Example~\ref{example:ex-of-bn-gen} we conclude that
generic\footnote{We will always mean that ``general'' = ``Zariski
general'' as long as the $\overline{\mathbb{Q}}$-varieties are
concerned.} surface $S\slash \fie$ is BN general. Now all the
previous constructions (of $E_r$, of the embedding inside
$\mathbb{P}^{12}$, etc.) run verbatim and we may assume that $S =
\Phi_{E_{3}}(S)$ (cf. Theorem~\ref{theorem:mukai}) is given over
$\fie$. Note however that we do not assert
$\mathrm{rank}\,\mathrm{Pic}\,S = 2$ here as in
Proposition~\ref{theorem:r-polarized} (the latter is actually
formulated for a \emph{very general} $S$ over $\mathbb{C}$).
Although one still has $\mathrm{Pic}\,S \ni H,C_0$ by
construction, where both $H$ and $C_0$ are over $\fie$ and
intersect as earlier.

Finally, the previous discussion applies to the surface $S_{22}$
from Remark~\ref{remark:s-12-is-unirational}, and we will identify
$S_{22}$ with a $\fie$-point varying over some Zariski open subset
$\subseteq\mathcal{K}_{12}$ (the choice of this point is yet to be
specified). Let us conclude this section with $2$ important
results:

\begin{proposition}
\label{theorem:my-potential-density} The set $S(\fie)$ of
$\fie$-points is Zariski dense in $S$. More precisely, after
possibly replacing $\fie$ with its finite extension, $S(\fie)$
contains a union $\displaystyle\bigcup_{i = 1}^{\infty}
E_{i}(\fie)$, where $E_i \subset S$ are algebraic curves such that
the sets $E_{i}(\fie)$ are infinite for all $i$ and
$\displaystyle\bigcup_{i = 1}^{\infty} E_{i}$ is Zariski dense in
$S$.
\end{proposition}

\begin{proof}
Note that $(H - 5C_{0})^2 = 0$. Then the surface $S$ caries the
structure of an elliptic fibration with general fiber $\sim H -
5C_{0}$ (see \cite[\S 3, Corollary
3]{piateckij-shapiro-shafarevich}). Now the result follows from
\cite[Section 4]{bogomolov-tschinkel}.
\end{proof}

\begin{theorem}[see \cite{manin}]
\label{theorem:manin-mordell-weil} Let $\frak{C}$ be a smooth
projective curve of genus $\ge 2$ over the function field
$\overline{\mathbb{Q}}(\mathbb{P}^{1})$. Suppose that the set
$\frak{C}(\overline{\mathbb{Q}}(\mathbb{P}^{1}))$ is infinite.
Then there exists a curve $\frak{C}_0$ over
$\overline{\mathbb{Q}}$ together with birational isomorphism over
$\overline{\mathbb{Q}}(\mathbb{P}^{1})$ between $\frak{C}$ and
$\frak{C}_0 \times \mathbb{P}^{1}$.
\end{theorem}

\bigskip

\section{Proof of Theorem~\ref{theorem:main}}
\label{section:the-proof-start}

We use the notation and conventions of
Section~\ref{section:preliminaries}. From now on all varieties are
assumed to be defined over $\overline{\mathbb{Q}}$ (unless stated
otherwise).

\refstepcounter{equation}
\subsection{}
\label{subsection:pr-2}

Let us begin with a setup. We are going to construct a special
pencil $\mathcal{S}$ of $\mathrm{K3}$ surfaces which will be used
later in our parameter count argument (see
Proposition~\ref{theorem:many-seFctions}). As a surprising
outcome, we will establish one intriguing property, valid for all
BN general $\mathrm{K3}$ surfaces of genus $12$ (see
Remark~\ref{remark:test-set-of-sections}).

Recall that equations of the surface $S \subset
G(3,7)\cap\mathbb{P}^{12}$ are represented by the data
$\Lambda,\Sigma_0 = (\sigma_1,\sigma_2,\sigma_3)$ (cf.
Theorem~\ref{theorem:mukai}), where $\Sigma_0\in \frak{M}_0$ is a
general $\fie$-point as in Remark~\ref{remark:s-12-is-unirational}
(cf. {\ref{subsection:pre-223}}). Similarly for $S_{22}$, one has
$\Sigma_{22} := (\tau_1,\tau_2,\tau_3)$, a general $\fie$-point in
$\frak{M}$. In addition, since the preimages of both $\frak{M}_0$
and $\frak{M}$ in $\mathcal{K}_{12}$ are projective bundles with
the typical fibers over $\Sigma_0$ and $\Sigma_{22}$ being (the
dual of) $\mathbb{P}^{13}\subset\mathbb{P}(H^{0}(G(3, 7),
\bigwedge^{3}\mathcal{E}_{3})) \ni \Lambda,\Lambda_{22}$, we may
assume that $\Lambda = \Lambda_{22} = H$ for some general
hyperplane $H \in H^0(G(3,7),\mathcal{O}_{G(3,7)}(1))$.

Note further that since $\frak{M}_0 \subset \frak{M}$ is of
codimension $1$, the skew-forms $\sigma_i$ also vary in
codimension $1$, which implies that the points
$(\tau,\sigma_2,\sigma_3),\tau\in H^{0}(G(3, 7),
\bigwedge^{2}\mathcal{E}_{3})$, are Zariski dense in $\frak{M}$.
In particular, one can take
\begin{equation}
\label{sigma-22-special} \Sigma_{22} = (\tau_{1}, \tau_{2},
\tau_{3}) := (\tau, \sigma_{2}, \sigma_{3}),
\end{equation}
where $\tau$ is generic.

Next we put $Z = \mathbb{P}^1$, with projective coordinates
$[t_0:t_1]$ and the function $t: = t_1/t_0$, and $G_Z := G(3,7)
\times Z$ together with the projection $\rho: G_Z \longrightarrow
Z$ on the second factor. One may identify (the general point of)
the scheme $G_Z \slash Z$ and the $\fie(t)$-variety $G(3,7)$. Let
us introduce the locus $\mathcal{V}_{22} := G_Z \cap (t_0\Sigma_0
+ t_1\Sigma_{22} = 0)$, for $\Sigma_{22}$ as in
\eqref{sigma-22-special}, so that
$$
\mathcal{V}_{22} \cap (\Lambda = 0) \cap (t = 0) = S \qquad
\mbox{and} \qquad \mathcal{V}_{22} \cap (\Lambda = 0) \cap (t =
\infty) = S_{22}
$$
by our setup. Consider also $\mathcal{S} := \mathcal{V}_{22} \cap
(\Lambda = 0)$ and the morphism $f := \rho\big\vert_{\mathcal{S}}
: \mathcal{S} \longrightarrow Z$ such that $f^{-1}(0) = S$ and
$f^{-1}(\infty) = S_{22}$. Again, with a slight abuse, we will
refer to (the general point of) $\mathcal{S} \slash Z$ as a
\emph{$\fie(t)$-surface}.

Finally, if $(\mathbb{P}^{12})^{\vee} \subset
\mathbb{P}(H^{0}(G(3, 7), \bigwedge^{3}\mathcal{E}_{3}))$ is the
space of all linear forms on $\mathbb{P}^{12} \supset S,S_{22}$,
then for a general $\lambda\in(\mathbb{P}^{12})^{\vee}(\fie)$ we
define $\mathcal{C} := \mathcal{S} \cap (\lambda = 0)$ (which
again will be sometimes referred to as a \emph{$\fie(t)$-curve}).
Note that $C := S \cdot \mathcal{C}$ is a general hyperplane
section$\slash\fie$ of $S$ and exactly the same constructions go
for $S_{22},\Sigma_{22}$, etc.

Fix an arbitrary general point $O \in
C(\overline{\mathbb{Q}})$.\footnote{This is only a technical
assumption. After all $O$ can be chosen arbitrary.} We want to
find a point $\frak{O} \in \mathcal{C}(\overline{\mathbb{Q}}(t))$
such that $O$ is the specialization of $\frak{O}$ at $t = 0$ (i.e.
$O = S \cdot \frak{O}$ on the surface $\mathcal{C} \longrightarrow
Z$). But before we proceed let us point out the next

\begin{lemma}
\label{theorem:pic-calc} The surface $S_{22}$ has geometric Picard
number $1$.\footnote{This is just a ``constructive'' reformulation
of the main result proved in \cite{ellenberg} (recall that this
result merely claims the existence of \emph{some} $\mathrm{K3}$
surface$\slash\overline{\mathbb{Q}}$ which has geometric Picard
number $1$). We provide this reformulation here because the
surface $S_{22}$ is the one we want to eventually construct (cf.
Theorem~\ref{theorem:main}).}
\end{lemma}

\begin{proof}
We follow the paper \cite{ellenberg}.

Let $\Omega\subset\mathbb{P}^{19}$ be the fundamental domain such
that $\mathcal{K}_{12} = \Omega/\Gamma$ for a subgroup $\Gamma$ in
the group of automorphisms of the lattice $H^{2}(S_{22},
\mathbb{Z})$. Let also $\mathcal{M}(\ell^{N})$ be a connected
component of the moduli space of all $\mathrm{K3}$ surfaces of
genus $12$ having both level $\ell^N$ and $p$ structures for some
$N \in \mathbb{N}$ and primes $p \ne \ell$. Namely, if
$\Gamma(p\ell^{N})\subset\Gamma$ is the corresponding congruence
subgroup, then $\mathcal{M}(\ell^{N}) :=
\Omega/\Gamma(p\ell^{N})$. Similarly, we define $\mathcal{M} :=
\Omega/\Gamma(p)$, a connected component of the moduli space of
all $\mathrm{K3}$ surfaces of genus $12$ with level $p$ structure.
Note that both $\mathcal{M}(\ell^{N})$ and $\mathcal{M}$ are
algebraic varieties over $\overline{\mathbb{Q}}$.

There is a Galois covering $\pi : \mathcal{M}(\ell^{N})
\longrightarrow \mathcal{M}$ with the Galois group
$\overline{\Gamma} := \Gamma(p)/\Gamma(p\ell^{N})$. Now the
arguments of \cite{ellenberg} imply that for every
$\overline{\mathbb{Q}}$-point $y$ from Zariski dense subset in
$\mathcal{M}$ the Galois group of $\pi^{-1}(y)$ over $y$ is also
$\overline{\Gamma}$. Then, again as in \cite{ellenberg}, the
$\mathrm{K3}$ surface $\frak{S}$ corresponding to $y$ has
geometric Picard number $1$. Moreover, since both
$\mathcal{K}_{12}$ and $\mathcal{M}$ are constructed as GIT
quotients of the Hilbert scheme of all $\mathrm{K3}$ surfaces of
genus $12$, we may assume that $y \in \mathcal{K}_{12}$ is a
general $\overline{\mathbb{Q}}$-point. We now put $S_{22} :=
\frak{S}$ to conclude the proof (recall that one is allowed to
extend the ground field $\fie$).
\end{proof}

\refstepcounter{equation}
\subsection{}
\label{subsection:pr-2as}

Let $U$ be the space of all $(3 \times 7)$-matrices of rank $3$
over $\overline{\mathbb{Q}}$ considered up to a scalar multiple.
The group $G := \text{PGL}(4,\overline{\mathbb{Q}})$ naturally
acts on $U$ and $G(3, 7) = U//G$ (cf. \cite[8.1]{mukai-moduli}).

In particular, any point in $G(3, 7)$ is the $G$-orbit of some
matrix $M \in U$, having the entries $m_{ij}$, $1 \le i \le 3,1
\le j \le 7$ (this also follows from the moduli interpretation of
$G(3,7)$). In this way, the equations $t_0\Sigma_0 +
t_1\Sigma_{22} = 0$, $\Lambda = 0$ and $\lambda = 0$ from
{\ref{subsection:pr-2}} lift to some $G$-invariant (homogeneous)
polynomial relations on $U$ between $m_{ij}$.

More generally, any section of the morphism $\rho: G_Z
\longrightarrow Z$, or equivalently, any point in $G(3,
7)(\overline{\mathbb{Q}}(t))$, yields a $1$-parameter algebraic
family of projective planes inside $\mathbb{P}^6$, which is the
same as a $\mathbb{P}^2$-bundle over $Z = \mathbb{P}^1$. This
bundle has a section, so that our family admits a lifting to $U$,
i.e. it is (generically) represented by some $M$ with $m_{ij} =
p_{ij}(t)$, where $p_{ij}(t)$ are polynomials in $t$.

Now, let us identify the above $\overline{\mathbb{Q}}$-point $O
\in C \subset G(3, 7)$ with a matrix in $U$, which can be assumed
to be of the form
$$
O = \begin{pmatrix} 1 & 0 & 0 & 0 & \ldots & 0\\
0 & 1 & 0 & 0 & \ldots & 0\\
0 & 0 & 1 & 0 & \ldots & 0
\end{pmatrix}
$$
after fixing a (projective) basis on $\mathbb{P}^6$. Choose also a
point $M \in G(3, 7)(\overline{\mathbb{Q}}(t))$ such that the
corresponding matrix has entries
\begin{equation}
\label{test-section} p_{ij}(t) := Y_{ij} + X_{ij}t
\end{equation}
for some $X_{ij}, Y_{ij} \in \overline{\mathbb{Q}}$ and all $1 \le
i \le 3,1 \le j \le 7$, where we additionally put $Y_{ij} := 0$
for all $j \ge 4$. Regard $X_{ij}, Y_{ij}$ as projective
coordinates on $\mathbb{P}^{29}$ (acted by $G$) and identify $M$
with the corresponding $\overline{\mathbb{Q}}(t)$-point in
$\mathbb{P}^{29}$. This $M$ is our candidate for the point
$\frak{O} \in \mathcal{C}(\overline{\mathbb{Q}}(t))$ mentioned at
the end of {\ref{subsection:pr-2}} (note that $M$ obviously
specializes to $O$ at $t = 0$):

\begin{proposition}
\label{theorem:many-seFctions} The set
$\mathcal{C}(\overline{\mathbb{Q}}(t))$ contains a point
$\frak{O}$ of the form \eqref{test-section}.
\end{proposition}

Before turning to the proof we introduce a bit more of notation.

\refstepcounter{equation}
\subsection{}
\label{subsection:pr-2asdfg}

Let $X$ (resp. $Y$) be the $(3 \times 7)$-matrix with entries
$X_{ij}$ (resp. $Y_{ij}$), $1 \le i \le 3,1 \le j \le 7$ (and
$Y_{ij} = 0$ for all $j \ge 4$). Let also $X_{j}$ (resp. $Y_j$),
$1 \le j \le 7$, be the $j$-th column of $X$ (resp. of $Y$). Then
we write
$$
\tau(X + Y) = \tau(Y) + \sum_{i = 1}^{3}\big(\sum_{j =
1}^{7}\beta^{(i)}_{j}X_{j}\big) \wedge Y_{i} + \tau(X)
$$
for some $\beta^{(i)}_{j} \in \fie$ and $\tau$ as in
\eqref{sigma-22-special}. Here the notation $\tau(X)$ means that
one treats $\tau$ as a skew-form on
$H^0(G(3,7),\mathcal{E}_3)^{\vee}$ of degree $2$ and evaluates it
on the pairs of columns of $X$ (same applies to $\tau(Y)$ of
course). This amplifies the lifting to $U$ of the equations on
$G_Z$ mentioned in {\ref{subsection:pr-2as}}.

Note that $X,Y$ are $3$-vectors with entries in $X_{ij},Y_{ij}$.
Denote by $Q_{1\alpha}(X, X)$ (resp. by $Q_{1\alpha}(Y, Y)$) the
$\alpha$-th component of the vector $\tau(X)$ (resp. of
$\tau(Y)$), $\alpha\in\{1,2,3\}$, and by $Q_{1\alpha}(X, Y)$ the
$\alpha$-th component of the vector $\displaystyle\sum_{i =
1}^{3}(\displaystyle\sum_{j = 1}^{7}\beta^{(i)}_{j}X_{j}) \wedge
Y_{i}$.

In the same way, one defines $Q_{(i + 1)\alpha}(X, X),Q_{(i +
1)\alpha}(Y, Y),Q_{(i + 1)\alpha}(X, Y)$ for all $\sigma_i$ from
{\ref{subsection:pr-2}}. Finally, we write
$$
\Lambda(X + Y) = \Lambda(Y) + \sum_{i, k = 1}^{3}\big(\sum_{j =
1}^{7}\beta^{(ik)}_{j1}X_{j}\big) \wedge Y_{i} \wedge Y_{k} +
\sum_{i = 1}^{3}\big(\sum_{j, k = 1}^{7}\beta^{(i)}_{jk1}X_{j}
\wedge X_{k}\big) \wedge Y_{i} + \Lambda(X)
$$
for some $\beta^{(i)}_{jk1}, \beta^{(ik)}_{j1} \in \fie$, and
similarly one introduces $\beta^{(i)}_{jk2}, \beta^{(ik)}_{j2}$
for $\lambda$. (Again, the notation $\Lambda(X),\lambda(X)$ stems
from interpreting both $\Lambda$ and $\lambda$ as skew-forms on
$H^0(G(3,7),\mathcal{E}_3)^{\vee}$ of degree $3$, applied to the
triples of columns $X_i$).

\begin{proof}[Proof of Proposition~\ref{theorem:many-seFctions}]
In the preceding notation, let us substitute $p_{ij}(t),1 \le i
\le 3$, $1 \le j \le 7$, from \eqref{test-section} into the
equation $\Sigma_0 + t\Sigma_{22} = 0$ and equate all the
coefficients of the resulting $9$ quadratic polynomials in $t$ to
zero. Then, since \eqref{sigma-22-special} takes place and $O \in
C$, we get a closed $G$-invariant subscheme in $\mathbb{P}^{29}$,
given by the equations
\begin{equation}
\label{eq-of-fr-f-red} \left\{
\begin{array}{l}
Q_{1\alpha}(X, X) = Q_{1\alpha}(X, Y) + Q_{2\alpha}(X, X) =
Q_{1\alpha}(Y, Y) + Q_{2\alpha}(X, Y) = 0,\\
\\
Q_{l\alpha}(X, X) = Q_{l\alpha}(X, Y) = 0,
\end{array}
\right.
\end{equation}
$\alpha\in\{1,2,3\},l \in \{3, 4\}$ (these are precisely the
conditions that $G \cdot M \in
\mathcal{V}_{22}(\overline{\mathbb{Q}}(t))$ for the locus
$\mathcal{V}_{22} \subset G_Z$ introduced in
{\ref{subsection:pr-2}}).

Further, after substituting $p_{ij}(t)$ from \eqref{test-section}
into the equations $\Lambda = \lambda = 0$ and equating all the
coefficients of the resulting $2$ cubic polynomials in $t$ to
zero, we arrive at a closed $G$-invariant subscheme in
$\mathbb{P}^{29}$, given by the equations
\begin{equation}
\label{eq-of-lam-big} \sum_{i, k = 1}^{3}\big(\sum_{j =
1}^{7}\beta^{(ik)}_{j\gamma}X_{j}\big) \wedge Y_{i} \wedge Y_{k} =
\sum_{i = 1}^{3}\big(\sum_{j, k =
1}^{7}\beta^{(i)}_{jk\gamma}X_{j} \wedge X_{k}\big) \wedge Y_{i} =
R_{\gamma}\big(X, X\big) = 0,
\end{equation}
$\gamma \in \{1, 2\}$, where $R_{\gamma}(X, X)$ is a linear
combination of $(3 \times 3)$-minors of the matrix $X$. Again,
both \eqref{eq-of-fr-f-red} and \eqref{eq-of-lam-big} yield $G
\cdot M \in \mathcal{C}(\overline{\mathbb{Q}}(t))$, so it remains
to satisfy these equations for one particular $X$.

Consider the rank $1$ matrix
\begin{equation}
\label{some-mat-x}
\begin{pmatrix} W_1l_1 & W_1l_2 & \ldots & W_1l_7\\
W_2l_1 & W_2l_2 & \ldots & W_2l_7\\
W_3l_1 & W_3l_2 & \ldots & W_3l_7
\end{pmatrix}
\end{equation}
for some $W_i,l_j \in \overline{\mathbb{Q}}$ and evaluate
\eqref{eq-of-fr-f-red}, \eqref{eq-of-lam-big} at $X :=$
\eqref{some-mat-x}, $Y := O$ (cf. the discussion in
{\ref{subsection:pr-2asdfg}}). Let $V$ be the linear space spanned
by $x_{1i} := W_1l_i,x_{2i} := W_{2}l_i,x_{3i} := W_{3}l_i,1 \le i
\le 7$. Then each of $Q_{l\alpha},l\in\{2,3,4\},\alpha \in
\{1,2,3\}$, turns into a linear form $H_{l\alpha}$ on the subspace
$V_{\alpha}\subset V$ given by equations
$$
x_{11} = \ldots = x_{33} = x_{i(\alpha)j} = 0,
$$
$4 \le j \le 7$, for some fixed $i(\alpha)$ with $\{i(1), i(2),
i(3)\} = \{1, 2 ,3\}$. In the same way, we get the forms
$H_{1\alpha} \in V$ and $H_{\Lambda}, H_{\lambda} \in
\displaystyle\sum_{\alpha} V_{\alpha}$, corresponding to
$Q_{1\alpha}$ and $\Lambda,\lambda$, respectively.

One may impose the following $2$ restrictions on
$H_{l\alpha},H_{\Lambda},H_{\lambda}$:

\begin{lemma}
\label{theorem:few-eqs-1} $Q_{12}(Y,Y) = Q_{13}(Y,Y) = 0$ (hence
$Q_{12}(Y,Y) \ne 0$) independently of $O \in S$.
\end{lemma}

\begin{proof}
Indeed, the section $X_1 \wedge X_2$ equals $(1,0,0)$ when
evaluated at $O$, and the same holds for all linear combinations
of $X_1 \wedge X_2$ and $\sigma_i$. We can take $\tau$ to be such
a combination (cf. the construction of $\Sigma_{22}$ in
\eqref{sigma-22-special}). Then $\tau = (1,0,0)$ on $S \cap (X_1
\wedge X_2 \wedge X_3 \ne 0)$ for the latter locus being lifted to
$U$ as usual (cf. {\ref{subsection:pr-2as}}).
\end{proof}

\begin{lemma}
\label{theorem:few-eqs-2} For a given general common zero $P$ of
linear forms $H_{31},H_{32},H_{4\alpha},\alpha\in\{1,2,3\}$,
varying together with $O \in S$, we have $H_{22}(P) = H_{23}(P) =
H_{33}(P) = H_{\Lambda}(P) = H_{\lambda}(P) = 0$ and $H_{21}(P)
\ne 0$.
\end{lemma}

\begin{proof}
Note that $X$ in \eqref{some-mat-x} can be replaced with any point
from the orbit $G^* \cdot X$ (for $\mathcal{S},\mathcal{C},O$
being fixed). Here $G^* := \text{PGL}(5,\overline{\mathbb{Q}})$
comes from the natural action on $U$. This induces an effective
$G^*$-action on the projectivization
$\mathbb{P}(\displaystyle\sum_{\alpha} V_{\alpha}) \ni P$. The
same holds also for $G = \text{PGL}(4,\overline{\mathbb{Q}})$ of
course.

In turn, the common zeroes of the forms
$H_{31},H_{32},H_{4\alpha},\alpha\in\{1,2,3\}$, constitute a
subspace $\mathbb{P}(V_0) \subset
\mathbb{P}(\displaystyle\sum_{\alpha} V_{\alpha})$ of dimension
$6$. Then the groups $G,G^*$ induce the
$\text{PGL}(7,\overline{\mathbb{Q}})$-action on $\mathbb{P}(V_0)$,
so that $\text{PGL}(7,\overline{\mathbb{Q}}) \cdot P =
\mathbb{P}(V_0)$.

Further, the common zero locus $\Xi$ of $H_{22},H_{23},H_{33},
H_{\Lambda}$ and $H_{\lambda}$ on $\mathbb{P}(V_0)$ is of
dimension $\ge 1$. Then, after acting by
$\text{PGL}(7,\overline{\mathbb{Q}})$, we can achieve $\Xi \ni P$
(recall again that in this case $X$ is just replaced by another
point from $G^* \cdot X$). Finally, since $\dim \Xi \ge 1$, the
form $H_{21}$ does not vanish at the general point of $\Xi$ (this
is easily seen by generality of $O,\sigma_i$, etc.), which we can
set to be $P$. This finishes the proof.
\end{proof}

By introducing an extra variable $W$, from \eqref{eq-of-fr-f-red},
\eqref{eq-of-lam-big}, \eqref{some-mat-x} and
Lemmas~\ref{theorem:few-eqs-1}, \ref{theorem:few-eqs-2} we obtain
a closed subscheme $\Gamma$ in $\mathbb{P}^7$ with projective
coordinates $l_i,W_j,W$, given by $9$ equations
\begin{equation}
\label{eq-of-gam} \left\{
\begin{array}{l}
W^{2} - H_1(W_1l_{4}, \ldots, W_{3}l_{7}) = H_2(W_1l_{4}, \ldots, W_{3}l_{7}) = \ldots = H_6(W_1l_{4}, \ldots, W_{3}l_{7}) = 0,\\
\\
F_{l}(W_1l_{1}, \ldots, W_{3}l_{7}) = 0,\\
\end{array}
\right.
\end{equation}
$l\in\{1,2,3\}$, for some linearly independent forms $H_{i}, F_{l}
\in H^{0}(\mathbb{P}^{20}, \mathcal{O}(1))$. (Here we have
employed the fact that $Q_{1\alpha}(Y,Y)\ne 0$ and that the
entries $W_il_j$ of \eqref{some-mat-x} are defined up to a scalar
multiple; this explains the $W^2$-term in \eqref{eq-of-gam}).

\begin{lemma}
\label{theorem:all-about-w} $W \ne 0$ identically on $\Gamma$.
\end{lemma}

\begin{proof}
The equations of \eqref{eq-of-gam} involving $H_2,\ldots,H_6$ are
equivalent to the following:
$$
P_i(W_1,W_2,W_3)l_i - Q_{i}(W_1,W_2,W_3)l_7 = P_1(W_1,W_2,W_3) =
P_2(W_1,W_2,W_3) = 0,
$$
$i\in\{4,5,6\}$, where $P_i,Q_i$ are some homogeneous forms. The
latter obviously have a common non-trivial solution in $W_i,l_j$,
hence so does $W^2 - H_1 = 0$. Then the remaining equations $F_l =
0$ in \eqref{eq-of-gam} become a system of $3$ equations
$L_i(W_1l_1,\ldots,W_3l_3) = c_i$ for some non-zero linear forms
$L_i$ and $c_i\in\overline{\mathbb{Q}}$. Again, the latter system
has a non-trivial solution in $W_i,l_j$, so that $W \ne 0$ on
$\Gamma$ as claimed.
\end{proof}

It follows from Lemma~\ref{theorem:all-about-w} that there exists
a matrix $X_0$ of the form \eqref{some-mat-x} such that the
corresponding point in $\mathbb{P}^{7}$ belongs to $\Gamma \cap (W
= 1)$. Then the pair $(X, Y) := (X_{0}, O)$ provides a common
solution to \eqref{eq-of-fr-f-red} and \eqref{eq-of-lam-big}. This
determines a point $\frak{O} \in G(3,
7)(\overline{\mathbb{Q}}(t))$ of the form \eqref{test-section} by
our previous discussion (cf. {\ref{subsection:pr-2as}}) and
because for $(X, Y) = (X_{0}, O)$ one gets a rank $3$ matrix in
\eqref{test-section} whenever $|t| \ll 1$. Thus we have $\frak{O}
\in \mathcal{C}(\overline{\mathbb{Q}}(t)(t))$ by construction
(with $O \in C$ being the specialization of $\frak{O}$ at $t =
0$).

Proposition~\ref{theorem:many-seFctions} is completely proved.
\end{proof}

\begin{corollary}
\label{theorem:many-sections-cor} The set
$\mathcal{C}(\overline{\mathbb{Q}}(t))$ is infinite.
\end{corollary}

\begin{proof}
Note that the proof of Proposition~\ref{theorem:many-seFctions}
(cf. Lemmas~\ref{theorem:few-eqs-1} and \ref{theorem:few-eqs-2})
does not depend on the choice of $O \in C(\overline{\mathbb{Q}})$
(notation is as earlier). Then we obtain
$\frak{O}\in\mathcal{C}(\overline{\mathbb{Q}}(t))$ for various $O$
and the claim follows.
\end{proof}

\begin{remark}
\label{remark:test-set-of-sections} Let us indicate that all
$\lambda$ with $\lambda(\frak{O}) = 0$ form a codimension $\le 2$
subspace $\frak{L}_{\frak{O}}\subset(\mathbb{P}^{12})^{\vee}$.
Indeed, the first (linear) constraint is $\lambda(O) = 0$ and the
second one is obtained from \eqref{eq-of-lam-big} after plugging
in $\gamma := 2,Y := O,X := X_0$.  In addition, all the preceding
arguments apply literally, if one replaces $O \in C = S \cap
(\lambda = 0)$ by $O \in S_{22} \cap (\lambda = 0)$. Note also
that Corollary~\ref{theorem:many-sections-cor} together with
Lemma~\ref{theorem:isomorphism-to-the-product} below imply all the
curves $S_{22} \cap (\lambda = 0)$ are isomorphic for \emph{all}
BN general $S_{22}\subset\mathbb{P}^{12}$ and any fixed general
$\lambda\in(\mathbb{P}^{12})^{\vee}$. This fact is not
particularly surprising from the Gromov-Witten theory view point.
Moreover, one could try to refine the existing curve counting
techniques for $\mathrm{K3}$ surfaces (see e.g.
\cite{maulik-et-all}) by observing that any generic $C \sim
L_{22}$, for example, does not change under the \emph{large
complex structure limit} of $S_{22}$ (cf. \cite{gross-wil}), since
the corresponding family of smooth curves is isotrivial according
to our reasoning. This equips, in some canonical way, any such $C
\subset S_{22}$ with certain \emph{modular structure} (by which we
simply mean a $3$-valent graph inscribed inside $\mathbb{P}^1 =$
the large limit of $C$, as discussed in \cite{thyu}, say), and so
the problem reduces to counting these discrete objects (compare
with \cite{fed-yu-monodromy}). Finally, it would be very
interesting to find out whether the same property of hyperplane
sections holds for other families of curves, or for BN general
polarized $\mathrm{K3}$ surfaces from other moduli spaces
$\mathcal{K}_g,g\ne 12$. This discussion, however, is beyond the
scope of the present paper.
\end{remark}

Further, since the (smooth projective) $\fie(t)$-curve
$\mathcal{C}$ is of genus $12$ (cf. the end of
{\ref{subsection:pr-2}}),
Corollary~\ref{theorem:many-sections-cor} and
Theorem~\ref{theorem:manin-mordell-weil} imply that $\mathcal{C}$,
when treated as a $\fie$-surface together with the morphism
$f\big\vert_{\mathcal{C}} : \mathcal{C} \longrightarrow Z$, is
birationally isomorphic over $Z$ to $\Gamma \times Z$ for some
$\overline{\mathbb{Q}}$-curve $\Gamma$. After a suitable
coordinate change on $Z$, we pick a Zariski open subset $U_Z
\subset Z$ such that $0, \infty \in U_Z$,
$f\big\vert_{\mathcal{C}}$ is smooth over $U_Z$ and
$(f\big\vert_{\mathcal{C}})^{-1}(t') \simeq \Gamma$ for all $t'
\in U_Z \setminus{\{0\}}$. Put also
$\mathcal{C}_{\scriptscriptstyle U_Z} := f^{-1}(U_Z) \cap
\mathcal{C}$.

\begin{lemma}
\label{theorem:isomorphism-to-the-product} The surfaces
$\mathcal{C}_{\scriptscriptstyle U_Z}$ and $C \times U_Z$ are
isomorphic over $U_Z$ (i.e. the family
$f\big\vert_{\mathcal{C}_{\scriptscriptstyle U_Z}} :
\mathcal{C}_{\scriptscriptstyle U_Z} \longrightarrow U_Z$ is
trivial).
\end{lemma}

\begin{proof}
Take a small disk $\Delta \subset U_Z$ around $0$ and consider the
period map $\zeta : \Delta \longrightarrow \frak{H}_{12}$,
associated with $f\big\vert_{\mathcal{C}}$, into the Ziegel
upper-half space $\frak{H}_{12}$. Here $\zeta$ is a holomorphic
morphism, constant on $\Delta\setminus{\{0\}}$, hence also on
$\Delta$, and so we get
$$
C = (f\big\vert_{\mathcal{C}})^{-1}(0) \simeq
(f\big\vert_{\mathcal{C}})^{-1}(t') \simeq \Gamma.
$$
Thus all fibers of $f\big\vert_{\mathcal{C}_{\scriptscriptstyle
U_Z}}$ are isomorphic to $C$. This implies that the family
$f\big\vert_{\mathcal{C}_{\scriptscriptstyle U_Z}} :
\mathcal{C}_{\scriptscriptstyle U_Z} \longrightarrow U_Z$ is
locally trivial in the analytic category. Moreover, since through
any general point on $C$ there passes a section of
$f\big\vert_{\mathcal{C}_{\scriptscriptstyle U_Z}}$ (cf.
Corollary~\ref{theorem:many-sections-cor}), the sheaf
$A_{\scriptscriptstyle U_Z}$ of relative automorphisms of
$f\big\vert_{\mathcal{C}_{\scriptscriptstyle U_Z}} :
\mathcal{C}_{\scriptscriptstyle U_Z} \longrightarrow U_Z$ is
constant. In particular, we have $H^{1}(U_Z, A_{\scriptscriptstyle
U_Z}) = 0$, which means that the family
$f\big\vert_{\mathcal{C}_{\scriptscriptstyle U_Z}} :
\mathcal{C}_{\scriptscriptstyle U_Z} \longrightarrow U_Z$ is
trivial.
\end{proof}

\begin{lemma}
\label{theorem:isomorphism-to-the-product-1} Through every point
on $\mathcal{C}_{\scriptscriptstyle U_Z} \subseteq \mathcal{C}$
there passes at most one section of $f\big\vert_{\mathcal{C}}$.
\end{lemma}

\begin{proof}
Indeed, otherwise there exists a rational dominant map
$\mathbb{P}^1 \dashrightarrow C$ induced by the projection
$\mathcal{C}_{\scriptscriptstyle U_Z} \simeq C \times U_Z
\longrightarrow C$ on the first factor (see
Lemma~\ref{theorem:isomorphism-to-the-product}), which is
impossible.
\end{proof}

\begin{lemma}
\label{theorem:all-sections-are-over-k} Let $\frak{O}$ be a point
from $\mathcal{C}(\overline{\mathbb{Q}}(t))$ such that its
specialization $O := S \cdot \frak{O}$ at $t = 0$ is a point in
$S(\fie)$. Then we have $\frak{O} \in \mathcal{O}(\fie(t))$.
\end{lemma}

\begin{proof}
The point $O$ is $\mathrm{Gal}(\overline{\mathbb{Q}}/\fie)$-fixed,
which implies that the section $\frak{O}$ is
$\mathrm{Gal}(\overline{\mathbb{Q}}/\fie)$-invariant (otherwise
one gets a contradiction with
Lemma~\ref{theorem:isomorphism-to-the-product-1}), and hence
$\frak{O} \in \mathcal{C}(\fie(t))$.
\end{proof}

\begin{remark}
\label{remark:every-thing-in-U} Consider a $\fie(t)$-curve
$\widehat{\mathcal{C}} \subset \mathcal{S}$ given by the equation
$\widehat{\lambda} = 0$ for some $\widehat{\lambda} \in
(\mathbb{P}^{12})^{\vee}(\fie)$. Then, in the previous notation,
the family $f\big\vert_{\widehat{\mathcal{C}}_{\scriptscriptstyle
U_Z}} : \widehat{\mathcal{C}}_{\scriptscriptstyle U_Z}
\longrightarrow U_Z$ is smooth for every $\widehat{\lambda}$ from
some Zariski open subset $\frak{L} \subset
(\mathbb{P}^{12})^{\vee}(\fie)$ (with $\lambda\in\frak{L}$), where
$\widehat{\mathcal{C}}_{\scriptscriptstyle U_Z} :=
\widehat{\mathcal{C}} \cap f^{-1}(U_Z)$. Moreover, the same
argument as in the proof of
Lemma~\ref{theorem:isomorphism-to-the-product} shows that the
family $f\big\vert_{\widehat{\mathcal{C}}_{\scriptscriptstyle
U_Z}} : \widehat{\mathcal{C}}_{\scriptscriptstyle U_Z}
\longrightarrow U_Z$ is trivial for every such
$\widehat{\lambda}$, and so the statements of
Lemmas~\ref{theorem:isomorphism-to-the-product-1},
\ref{theorem:all-sections-are-over-k} also hold for
$\widehat{\mathcal{C}}_{\scriptscriptstyle U_Z}$ and
$\widehat{\mathcal{C}}$, respectively.
\end{remark}

Let $\frak{U}$ be the set of all points
$\in\mathcal{S}(\overline{\mathbb{Q}}(t))$ as in
Proposition~\ref{theorem:many-seFctions}, with $O$ varying through
$S(\fie)$, $\lambda$ varying through $\frak{L}$ (see
Remark~\ref{remark:every-thing-in-U}). Then from
Proposition~\ref{theorem:my-potential-density} and
Lemma~\ref{theorem:all-sections-are-over-k} we obtain that
$\frak{U} \subset \mathcal{S}(\fie(t))$ is Zariski dense in
$\mathcal{S}$. Thus $\mathcal{S}$ provides an example of
\emph{non-isotrivial} $\mathrm{K3}$ surface over the function
field $\fie(t)$, such that the set of $\mathcal{S}(\fie(t))$ is
Zariski dense, the genus of $\mathcal{S}$ is $12$, and the
geometric Picard number of $\mathcal{S}$ equals $1$ (the latter
follows from Lemma~\ref{theorem:pic-calc}). We can, however,
extract more from the construction of $\mathcal{S}$.

\refstepcounter{equation}
\subsection{}

Firstly, for every $\frak{O} \in \frak{U}$ we have $S_{22} \cdot
\frak{O} \in S_{22}(\fie)$, since both $\frak{O}$ and $S_{22}$ are
defined over $\fie$, hence both are
$\mathrm{Gal}(\overline{\mathbb{Q}}/\fie)$-invariant. Let us
consider the map $h : \frak{U} \longrightarrow S_{22}(\fie)$ given
by $\frak{O} \mapsto S_{22} \cdot \frak{O}$ for all $\frak{O}$.
Put $\frak{U}^* := h(\frak{U})$. Note that the set $\frak{U}^*$ is
infinite (for otherwise $\frak{L}\subset(\mathbb{P}^{12})^{\vee}$
would be of codimension $1$).

Further, after shrinking $\frak{U}$ if necessary, we get the
following:

\begin{lemma}
\label{theorem:1-to-1-lemma} The map $h$ is $1$-to-$1$ onto its
image.
\end{lemma}

\begin{proof}
Recall that according to Remark~\ref{remark:test-set-of-sections}
we can switch the roles between $S$ and $S_{22}$. Let us choose
then some general point $O \in \frak{U}^*$ and consider the
codimension $2$ subspace $\frak{L}_{\frak{O}}\subset
(\mathbb{P}^{12})^{\vee}(\fie)$ of all $\frak{O} \in \frak{U}$ for
which $h(\frak{O}) = S_{22} \cdot \frak{O} = O$. Now, since
$\frak{U}^*$ is infinite, we may assume that $\lambda \in
\frak{L}_{\frak{O}} \cap \frak{L}$. Note also that
Lemmas~\ref{theorem:isomorphism-to-the-product-1} and
\ref{theorem:all-sections-are-over-k} imply the definition of
$h(\frak{O})$ is independent of any other $\widehat{\lambda}\in
\frak{L}_{\frak{O}}$. Hence, for varying $\widehat{\lambda}$, we
simply take $\frak{O}$ to be the common point (from
$\mathcal{S}(\fie(t))$) of all $\mathcal{S}\cap (\widehat{\lambda}
= 0)$. This gives the $1$-to-$1$ statement.
\end{proof}

\begin{remark}
\label{remark:sect} Let us denote again by $\mathcal{S}$ and
$\frak{O}$ the images in $G(3, 7)$ of (the $\fie$-threefold)
$\mathcal{S}$ and (the $\fie$-curve) $\frak{O} \in\frak{U}$,
respectively, under the projection $G_Z = G(3, 7) \times Z
\longrightarrow G(3, 7)$. Then (the proof of)
Lemma~\ref{theorem:1-to-1-lemma} shows that $\mathcal{S} \cap \Pi
= \frak{O}$ for a codimension $12$ subspace $\Pi \subset
\mathbb{P}(\bigwedge^{3}H^{0}(G(3, 7),\mathcal{E}_{3})^{\vee})$
complementary to $\mathbb{P}^{12}\supset S,S_{22}$.
\end{remark}

\begin{lemma}
\label{theorem:1-to-2-lemma} The set $\frak{U}^*$ is Zariski dense
in $S_{22}$.
\end{lemma}

\begin{proof}
Let $E_i$ be as in Proposition~\ref{theorem:my-potential-density}.
Then by Remark~\ref{remark:sect} the Zariski closure of the set
$\{\frak{O} \in \frak{U} \ \big\vert \ O \in E_i\}$ in
$\mathcal{S}$ determines a cycle $E'_i\in
H^{1,1}(S_{22},\mathbb{Z})$ such that the sets $E'_i \cap
S_{22}(\fie)$ are infinite for all $i$ (see
Lemma~\ref{theorem:1-to-1-lemma}).

Note that the above family $f : \mathcal{S} \longrightarrow Z$ is
differentially trivial over $U_Z$. Then the induced diffeomorphism
between $S$ and $S_{22}$ maps every $E_i$ onto the corresponding
$E'_i$ by the monodromy invariance of all section $\frak{O}$ of
$f$. In particular, the union $\displaystyle\bigcup_{i =
1}^{\infty} E'_i$ is not contained in any curve on $S_{22}$, and
hence it is Zariski dense. Moreover, since the sets $E'_i \cap
S_{22}(\fie)$ are infinite for all $i$, we obtain that
$\displaystyle\bigcup_{i = 1}^{\infty} E'_i \cap S_{22}(\fie)
\subseteq \frak{U}^*$ is Zariski dense in $S_{22}$. Indeed,
otherwise all $E'_i \subseteq$ some \emph{algebraic} curve, a
contradiction.
\end{proof}

Lemmas~\ref{theorem:1-to-2-lemma} and \ref{theorem:pic-calc}
complete the proof of Theorem~\ref{theorem:main}.

\end{document}